\documentclass[
final,5p,twocolumn
]{elsarticle}%
\usepackage{amssymb, amsmath}
\usepackage{overpic}
\usepackage{float}

\usepackage{hyperref}  

\usepackage{xcolor}
\usepackage{pdfsync}

\usepackage{graphicx}
\graphicspath{{figures/}}
\usepackage{epsfig}
\usepackage{amsmath}
\usepackage{amssymb}
\usepackage{amsthm}
\usepackage{mathtools}
\usepackage{amsfonts}
\usepackage{lineno}
\usepackage{tikz}
\usetikzlibrary{positioning,shapes}
\usetikzlibrary{arrows}

\allowdisplaybreaks

\newtheorem{theorem}{Theorem}
\newtheorem{definition}[theorem]{Definition}

\newtheorem{lemma}[theorem]{Lemma}

\newtheorem{proposition}[theorem]{Proposition}
\newtheorem{remark}[theorem]{Remark}

\newtheorem{example}[theorem]{Example}
\newtheorem{problem}[theorem]{Problem}
\newtheorem{algorithm}[theorem]{Algorithm}

\newcommand{\abs}[1]{\left|#1\right|}

\def\field#1{\mathbb #1}%
\def\R{\field{R}}%
\def\N{\field{N}}%
\newcommand{\X}{\ensuremath{\mathcal X}}

\newcommand{\U}{\ensuremath{\mathcal U}}

\newcommand{\Rn}[1][n]{\R^{#1}}
\newcommand{\Rp}{\R_{\geq 0}}
\newcommand{\Rsp}{\R_{> 0}}

\def\K{\mathcal{K}}%
\DeclareMathOperator{\id}{id}

\def\KL{\mathcal{KL}}%
\def\Kinf{\mathcal{K}_\infty}%
\let\ol=\overline%
\let\ul=\underline%

\def\A{\mathcal{A}}

\usepackage{color}

\usepackage{xspace}	
\providecommand{\fsCLF}{fs-CLF\xspace}

\providecommand{\fsCLFs}{fs-CLFs\xspace}
\providecommand{\CLF}{CLF\xspace}

\providecommand{\CLFs}{CLFs\xspace}

\providecommand{\afcl}{\textcolor{black}{admissible finite-step feedback}\xspace}
\providecommand{\fslfcs}{{finite-step Lyapunov function candidates}\xspace}

\title{Control of discrete-time nonlinear systems via finite-step control Lyapunov functions}

\author[First]{Navid~Noroozi}
\address[First]{Otto-von-Guericke University of Magdeburg, Laboratory for Systems Theory and Automatic Control, 39106 Magdeburg, Germany}
\ead{navid.noroozi@ovgu.de}

\author[Second]{Roman~Geiselhart}
\address[Second]{University of Ulm, %
    Institute of Measurement, Control and Microtechnology, %
    Albert-Einstein-Allee 41, %
    89081 Ulm, Germany}
\ead{roman.geiselhart@uni-ulm.de} 

\author[Third]{Lars~Gr\"une}
\address[Third]{University of Bayreuth, 
    Mathematical Institute, %
    Universit\"atsstra\ss e 30, %
    95440 Bayreuth, Germany}
\ead{lars.gruene@uni-bayreuth.de}

\author[Fourth]{Fabian~R.~Wirth}
\address[Fourth]{University of Passau, %
    Faculty of Computer Science and Mathematics, %
    Innstra\ss e 33, %
    94032 Passau,    Germany}
\ead{fabian.lastname@uni-passau.de}

\cortext[corauth]{The work of N. Noroozi was supported by the Alexander von Humboldt Foundation.}

\begin{document}

\begin{abstract}
In this work, we establish different control design approaches for discrete-time systems, which build upon the notion of \emph{finite-step} control Lyapunov functions (\fsCLFs).
The design approaches are formulated as optimization problems and solved in a model predictive control (MPC) fashion.
In particular, we establish contractive multi-step MPC with and without reoptimization  and compare it to classic MPC.
The idea behind these approaches is to use the fs-CLF as running cost.
These new design approaches are particularly relevant in situations where information exchange between plant and controller cannot be ensured at all time instants.
An example shows the different behavior of the proposed controller design approaches.

\end{abstract}


\maketitle

\section{Introduction}

Lyapunov functions are a central tool in the context of nonlinear control theory as they do not only serve as certificates of stability and simplify stability proofs, but also provide means to quantify robustness or redesign the controller to improve robustness of the feedback connection~\cite{Khalil.2002}.
This has the drawback that systematic methods for obtaining Lyapunov functions for general nonlinear systems still do not exist.
In particular, standard Lyapunov function candidates, including quadratic, weighted supremum norm and weighted $1$-norm functions, do not necessarily decay at each time step.\footnote{Here we consider discrete-time systems. 
A similar conclusion also holds for continuous-time systems.}

In contrast to classic Lyapunov functions, so-called \emph{finite-step} Lyapunov functions are energy functions which do not have to decay at each time step, but only after a fixed and finite number of steps.
This relaxation leads to significant contributions in the context of stability analysis of (large-scale) nonlinear systems~\cite{Geiselhart.2014c,Geiselhart.2015,Gielen.2015,Geiselhart.2016,Noroozi.2018.SG}.
In particular, it has been shown that any proper scaling of a $p$-norm function is a finite-step Lyapunov function for a large class of \emph{asymptotically} stable nonlinear systems~\cite{Geiselhart.2014c}.
Such \emph{converse} Lyapunov theorems are \emph{constructive} for control purposes in the sense that they provide an explicit way of construction of a Lyapunov function for control systems.
This motivates the use of such results for the controller design in nonlinear control systems.
In this paper, we generalize the notion of finite-step Lyapunov functions to control systems by introducing the notion of finite-step \emph{control} Lyapunov functions (\fsCLF).

Given a \fsCLF, we reformulate a \fsCLF-based control design into an optimization problem.
In particular, we link the \fsCLF-based control design to model predictive control (MPC) approaches.
By considering three different optimization setups for the \fsCLF-based design, we come up with three \fsCLF-based MPC approaches: a) contractive multi-step MPC; b) contractive \emph{updated} multi-step MPC; and c) classic (i.e. one-step) MPC.
In c), we focus on MPC without terminal constraints and/or costs, see, e.g.~\cite[Section 7.4]{Grune.2017a} and the references therein for a thorough discussion on MPC with or without additional stabilizing terminal ingredients.
In a) and b), the optimization problem includes a contractive condition guaranteeing a decay rate after a finite number of time steps.
In all these schemes the \emph{running} cost in the respective optimization problem is taken as the \fsCLF.
The a priori knowledge of such a \fsCLF is guaranteed by the converse
Lyapunov theorem stated as Theorem~\ref{thm:converse-Lyap} below.

\begin{figure}[!ht]
\centering
\includegraphics[scale=0.8]{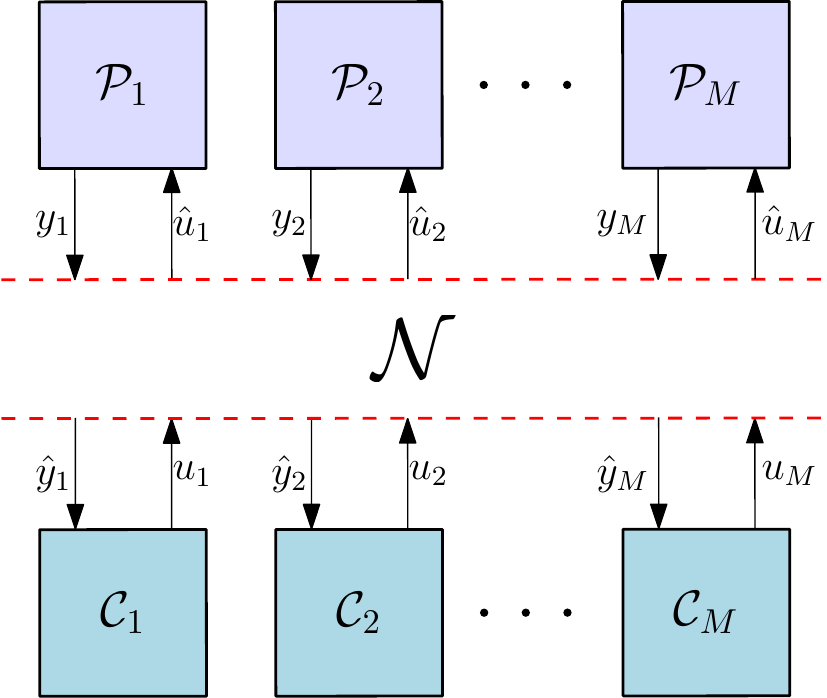}
\vspace{-0.35cm}\centering
\caption{The implementation of $M$ independent control loops over a single communication channel $\mathcal{N}$.}
\label{fig:NCS}
\end{figure}

Classic MPC approaches are based on the following philosophy: at \emph{each} time step we measure the current state value of the system, optimize a cost over the control input using model based predictions of the system response over a \emph{fixed} optimization horizon, implement the \emph{first} component of the computed control sequence, and repeat these steps ad infinitum~\cite{Grune.2017a}.
However, in practice, the controller and the plant may not communicate with each other at each time step.
In networked control systems applications, multiple (physically decoupled) plants often need to share a communication channel for exchanging information with their corresponding remotely located controllers; see~Figure~\ref{fig:NCS}.
Therefore only a few plants can exchange information with their
controllers at any instant of time and the remaining plants operate in
open-loop until they are granted access to the communication channel.
The allocation (also known as scheduling) of communication resources is
frequently performed in a \emph{periodic} fashion.
In such a situation, we need to develop a control setting in which the controller (if possible) sends not only one component, but a control sequence of the length equal to the periodicity of the allocation process at each transmission instant; see~\cite{Varutti.2009,Grune.2009} for such networked control systems configurations.
Such a scenario motivates our contractive \emph{multi-step} MPC, where the MPC does not communicate with the plant at each time step, but only after a fixed number of time steps.
The whole optimal control sequence is sent to the plant to compensate for the lack of access to the network.
To guarantee the stability of the resulting system, we include a \emph{contractive} constraint, which is obtained from the corresponding \fsCLF, into the optimization problem as an inequality constraint.

Another problem with the classic implementation of MPC is that the execution of the optimization problem at each time step may result in high computational cost (here we assume the controller and the plant can communicate at each time step).
To keep the computational cost low, inspired by~\cite{Grune.2015,Grune.2014b}, the second control scheme proposes an updating approach based on re-optimizations on \emph{shrinking} horizons which are computationally less expensive than re-optimizations on the full horizon in classic MPC schemes.
Similar to the first scheme, a \emph{contractive} constraint is also used in the optimization problem to ensure the stability of the overall system.
Finally, the third control scheme proposes a classic MPC approach in which the optimization problem is solved over a \emph{fixed} optimization horizon at each time step; and hence only the first component of the computed control sequence is applied to the plant.
Moreover, the contraction condition is \emph{not} considered as an additional inequality constraint in the optimization problem.
The absence of the contractive constraint reduces the computational complexity, though considering a fixed optimization horizon will increase the computational burden.

The MPC schemes we are proposing have similarities with MPC schemes known from the literature. Particularly, the scheme from Algorithm \ref{alg:multi-step-mpc}, in which the whole open loop optimal control sequence is used, is an instance of a contractive MPC scheme, as investigated, e.g., in~\cite{Alamir.2017a,Kothare.2000,Wan.2007,Yang.1993}. The contractivity constraint in Problem \ref{prob:OCP-1} can be seen as a nonlinear version of the respective condition in \cite{Kothare.2000,Wan.2007,Yang.1993}. Actually, under suitable conditions the explicit use of a contractive constraint can be replaced by a term in the cost functional with sufficiently high weight, see \cite[Theorem 3.18]{Worthmann.2011} or \cite{Alamir.2017a}. The paper \cite{Alamir.2017a} already mentions the possibility to use a \fsCLF in contractive MPC. Contractivity assumptions have also been used in MPC schemes with additional terminal constraints, see~\cite{Hanema.2017,Lazar.2015}. The updating technique with shrinking horizon in Algorithm \ref{alg:re-opt-step-mpc} was inspired by~\cite{Grune.2015,Grune.2014b}, where a theoretical robustness analysis of this method is performed. Finally, the MPC Algorithm \ref{alg:classic-mpc} without terminal constraints is classical and the particular stability analysis in Theorem \ref{thm:OCP-3a} uses techniques from \cite{Gruene.2009,Gruene.2010} (see also \cite[Section 6]{Grune.2017a}), which in turn can be seen as a refinement of earlier, similar approaches in \cite{Grimm.2005,Tuna.2006,Gruene.2008}.

Throughout this paper we consider a stabilization problem with respect to a \emph{closed} (not necessarily compact) set.
This treatment enables us to formulate several stabilization problems in a unified manner.

This paper is organized as follows:
First, relevant notation is recalled in Section~\ref{sec:notation}.
Then the notion of \fsCLFs together with some other relevant notions are introduced in Section~\ref{sec:preliminaries}.
The \fsCLF-based MPC schemes are developed in Section~\ref{sec:fsclf-based-mpc-approaches}.
Section~\ref{sec:conclusions} concludes the paper.

\section{Notation} \label{sec:notation}

In this paper, $\Rp (\Rsp)$ and $\N (\N^*)$ denote the nonnegative
(positive) real numbers and the nonnegative (positive) integers,
respectively.
For a set $\mathcal{S} \subseteq \Rn$, $\operatorname{int}(\mathcal{S})$ and $\operatorname{co}(\mathcal{S})$, respectively, denote the interior and the convex hull of $\mathcal{S}$.
Given $\mathcal{S} \subseteq \Rn$, $\mathcal{S}^\ell :=
\underbrace{\mathcal{S} \times \dots \times \mathcal{S}}_{\ell \, \,
  \mathrm{times}}$ is the $\ell$-fold Cartesian product.
The $i$th component of $v \in \Rn$ is denoted by $v_i$.
For any $x \in \Rn$, $x^\top$ denotes its transpose.
We write $(x,y)$ to represent $[x^\top,y^\top]^\top$ for $x \in \Rn,y \in \R^p$.
For $x \in \Rn$, we, respectively, denote the Euclidean norm and the maximum norm by $\abs{x}$ and by $\abs{x}_\infty$.
Given a nonempty set $\mathcal{A}\subset\Rn$ and any point $x\in\Rn$, we denote
$\abs{x}_\A \coloneqq \inf\limits_{y \in \A}\abs{x-y}$.
A function $\rho \colon \Rp \to \Rp$ is positive definite if it is continuous, zero at zero and positive otherwise. A positive definite function $\alpha$ is of class-$\K$ ($\alpha \in \K$) if it is zero at zero and strictly increasing. 
It is of class-$\Kinf$ ($\alpha \in \Kinf$) if $\alpha \in \K$ and also $\alpha(s) \to \infty$ if $s \to \infty$. A continuous function $\beta \colon \Rp \times \Rp \to \Rp$ is of class-$\mathcal{KL}$ ($\beta \in \mathcal{KL}$), if for each $s \geq 0$, $\beta(\cdot,s)\in\mathcal{K}$, and for each $r \geq 0$, $\beta (r,\cdot)$ is decreasing with $\beta (r,s)\to 0$ as $s \to \infty$.  The interested reader is referred to \cite{Kellett.2014} for more details about comparison functions.
The identity function is denoted by $\id$.
Composition of functions is denoted by the symbol $\circ$ and repeated composition of, e.g., a function $\gamma$ by $\gamma^{i}$.
For positive definite functions $\alpha,\gamma$ we write $\alpha<\gamma$ if $\alpha(s)<\gamma(s)$ for all $s>0$.

\section{Preliminaries}\label{sec:preliminaries}

We first introduce the notions of admissible finite-step feedback control laws and \fsCLFs.
We then show that our definition implies that an admissible finite-step feedback control law generated by a \fsCLF stabilizes the system of interest.
The idea how to construct a \fsCLF is given by a converse Lyapunov theorem.

\subsection{Finite-step control Lyapunov function}

Consider the discrete-time system
\begin{equation}  \label{eq:uncontrolled-system}
x (t+1) = g \big( x(t),u(t) \big), \qquad t \in \N
\end{equation}
with state $x \in \X \subseteq \R^n$ and control input $u \in \U \subseteq \R^m$.
We assume $g \colon \X \times \U \to \R^n$ is continuous.
Moreover, we assume that $g$ is \emph{$\K$-bounded} on $(\X,\U)$ as defined below.

\begin{definition}
A continuous and nonnegative $\omega \colon \Rn \to \Rp$ is called a
\emph{measurement function}, if the preimage of $0$ satisfies $\omega^{-1}(0)\neq\emptyset$.
\end{definition}
\begin{definition} \label{def:gKb} 
Consider system \eqref{eq:uncontrolled-system}.
Given measurement functions $\omega_1:\R^n\to\Rp$ and $\omega_2:\R^m\to\Rp$, we call $g$ \emph{$\K$-bounded} on $(\X,\U)$ with respect to $(\omega_1,\omega_2)$ if there exist $\kappa_i \in \K$, $i =1,2$ such that
\begin{align*}
\omega_1 (g(\xi,\mu)) \leq \kappa_1 (\omega_1(\xi)) + \kappa_2 (\omega_2(\mu)) 
\end{align*}
for all $\xi \in \X$ and all $\mu \in \U$.
\end{definition}

The concept of $\K$-boundedness was introduced in~\cite{Geiselhart.2014c}
for the case when $\omega_i(\cdot)=\abs{\cdot}$. Extensions to $\K$-boundedness with respect to one (resp. two measurement functions) are given in~\cite{Noroozi.2018.SG} (resp.~\cite{Geiselhart.2017}).
Here we extend this concept to the constraint sets $\mathcal{X}$ and $\mathcal{U}$.
Frequently, $\omega_2$ will be taken as a norm.
Note that in the classic case $\omega_i(\cdot)=\abs{\cdot}$, $\K$-boundedness is equivalent to continuity of $g$ in the origin and boundedness of $g$ on bounded sets, see~\cite[Lemma~5]{Geiselhart.2017}.
Thus, any closed-loop system, consisting of a continuous plant controlled by an optimization-based or quantized controller, is
$\K$-bounded. 
We note that $\K$-boundedness is a necessary condition for input-to-state stability, see~\cite[Remark~3.3]{Geiselhart.2016}.

Let ${\bf u}=(u(0), u(1), \ldots)$ denote a possibly infinite control sequence for system~\eqref{eq:uncontrolled-system}, where $u(i)\in \U$ for all $i=0,1,\ldots$.
If we only study trajectories of~\eqref{eq:uncontrolled-system} over a finite horizon, we might restrict to finite control sequences denoted by $\mathbf{u}_k := (u(0),\dots,u(k-1)) \in \U^k$.
Given a control sequence $\bf u$ and an initial value $\xi \in \X$, the corresponding solution to~\eqref{eq:uncontrolled-system} is denoted by $x (\cdot,\xi,\mathbf{u}(\cdot))$, also the notation $x (\cdot,\xi,\mathbf{u})$ or $x(\cdot)$ will be used.

We require some notation to state the definitions below.
Let $M \in \N^*$ be fixed. 
For $\xi\in \X$ and $\mathbf{u}_M = (u(0),  \ldots, u(M-1))\in \U^M$ we define
\begin{equation*}
g^1(\xi,\mathbf{u}_M) \coloneqq g(\xi,u(0))
\end{equation*}
and inductively, for $j=1,\ldots ,M-1$,
\begin{equation*}
g^{j+1}(\xi,\mathbf{u}_{M}) \coloneqq g\big(g^{j}(\xi, \mathbf{u}_M),u(j)\big).    \end{equation*}

We note that strictly speaking $g^j$ is only a function of $\xi$ and
$(u(0), \ldots,u(j-1))$, but there is no benefit in making this
precise notationaly so we stick to the simpler version.
Consider system~\eqref{eq:uncontrolled-system} and a map $q: \X \to
\U^M$. We wish to interpret $q$ as a feedback evaluated every $M$ steps.
Given an initial condition $x(0) = \xi \in \X$, the feedback $q$ determines a closed-loop trajectory $x_q$ of~\eqref{eq:uncontrolled-system} as
follows. For $j=0,\ldots,M-1$ we let $x_q(j+1) = g^{j+1}(x(0), q(x(0))$ and
at time $M$ we evaluate the feedback again and repeat the process. We
obtain inductively for $k\in \N, j=0,\ldots,M-1$ that
\begin{equation*}
    x_q(k M+ j+1) = g^{j+1} \bigl(x_q(k M), q(x_q( k M))\bigr).
\end{equation*}
In the sequel we use the notation $u_q \in \U^\N$ to denote the sequence of control inputs generated by the repeated application of the feedback $q$ and we denote interchangeably 
\begin{equation}\label{eq:interchangenotation}
    x(\cdot,\xi, u_q) = x_q (\cdot) .
\end{equation}

\begin{definition}\label{def:afcl}
Let a measurement function $\omega:\Rn \to\Rp$ and some $M\in \N^*$ be given.
A map $q:\X\to \U^M$ is called an \emph{\afcl} (of length $M$) for system~\eqref{eq:uncontrolled-system}, if for all $\xi \in \X$ and all $j =1,\ldots,M$, the following properties hold:
\begin{enumerate}
\item $g^j \big(\xi, q(\xi) \big)\in\X$;
\item $\xi\mapsto g^j \big(\xi, q(\xi) \big)$ is $\K$-bounded on $\X$ with respect to $\omega$, i.e., there exist $\kappa_j \in \K$ such that
\begin{align}\label{eq:K-bounded-closed-loop}
\omega \Big(g^j (\xi,q(\xi))\Big) \leq \kappa_j\big(\omega(\xi)\big) \qquad \forall \xi \in \X .
\end{align}
\end{enumerate}
\end{definition}

Condition~(i) of Definition~\ref{def:afcl} justifies the terminology \emph{admissible} as it ensures that trajectories of the closed-loop system obtained by applying the map $q$ to~\eqref{eq:uncontrolled-system}
stay in $\X$.
In addition, condition~(ii) ensures that along trajectories the measure $\omega$ remains bounded on bounded time intervals.

\begin{definition} \label{def:ISS-stabilization} Consider system
    \eqref{eq:uncontrolled-system} and let a measurement function
    $\omega:\Rn \to\Rp$ and some $M\in \N^*$ be fixed.  Consider an \afcl $q
    \colon \X \to \U^M$ for system~\eqref{eq:uncontrolled-system}.  We say
    that $q$ asymptotically $\omega$-stabilizes the set ${\cal A} := \omega^{-1}(0)$, if there
    exist $\beta \in \mathcal{KL}$ and $\gamma \in \K$ such that for all
    $\xi \in \X$ and all $t \in \N$ we have
\begin{equation} \label{eq:asymp-estimate}
\omega\big(x (t,\xi,u_q)\big) \leq \beta\big(\omega(\xi), t \big) .
\end{equation}
In this case, the resulting closed-loop system
\begin{equation}  \label{eq:controlled-system}
x_q (t+1) = g \big(x_q (t),u_q (t) \big), \qquad t \in \N,
\end{equation}
is asymptotically $\omega$-stable in ${\cal A}$. If the
function $\beta$ in~\eqref{eq:asymp-estimate} can be taken as
\begin{align}  \label{eq:beta-exp}
\beta (r,s) = C \sigma^s r,
\end{align}
with $C \geq 1$ and $\sigma \in [0,1)$, then we call $q$ \emph{exponentially} $\omega$-stabilizing.
\end{definition}

Note that standard asymptotic stability of the origin is obtained by taking the
measurement function $\omega(\cdot) = \abs{\cdot}$.

\begin{remark}
    We note that while the definition of the concept of $\omega$-stabilization looks familiar, some care has to be applied in  its interpretation.
    As the notion of a measurement function is quite general and as we do not assume continuity of the closed-loop system several surprising effects can appear.
    In particular, in the generality of Definition~\ref{def:ISS-stabilization} the following situations cannot be ruled out:
    \begin{enumerate}[(i)] 
      \item ${\cal A}$ is compact and all trajectories not starting in
        ${\cal A}$ diverge to $\infty$ or to the boundary of ${\cal
          X}$. This requires discontinuity of $q$.
      \item The feedback $q$ is continuous, ${\cal A}$ is unbounded and
        for certain trajectories $\mathrm{dist}_{\cal A}(x_q(\cdot,x_0))$ is
        strictly increasing.
      \item Given $\varepsilon>0$ there is no $\delta> 0$ such that
        $\mathrm{dist}_{\cal A}(x_0)<\delta$ implies $\mathrm{dist}_{\cal
          A}(x_q(t,x_0))< \varepsilon$ for all $t\geq 0$.
    \end{enumerate}
    Examples for these effects are easy to construct and we leave the details
    to the reader. There are easy additional assumptions that remove these
    peculiarities. For instance, one could assume that there is $\alpha\in
    {\cal K}$ such that $\alpha(\mathrm{dist}_{\cal A}(x)) \leq \omega(x)$
    for all $x\in {\cal X}$. This assumption already rules out (i) and
    (ii). 
\end{remark}

Now we introduce \emph{finite-step control Lyapunov functions}, which is the key concept used for the control design in the next section.

\begin{definition}\label{def:fsCLF}
Let $\ul\alpha,\ol\alpha\in\Kinf$, $M\in \N^*$ and $\omega:\Rn\to\Rp$ be a measurement function.
Consider a continuous function $V \colon \R^n \to \Rp$ satisfying for all $\xi\in\R^n$,
\begin{gather}
\ul\alpha (\omega(\xi)) \leq V(\xi) \leq \ol\alpha (\omega(\xi))  \label{eq:3}.
\end{gather}
The function $V$ is called a \emph{finite-step control Lyapunov function}
(\fsCLF) (for the time step $M$) for system~\eqref{eq:uncontrolled-system} if there exists an \afcl $q \colon \X \to \U^M$ for~\eqref{eq:uncontrolled-system} and a function $\alpha\in\Kinf$, $\alpha<\id$ such that for all $\xi\in\X$,
\begin{equation}\label{eq:decayV}
V(x (M,\xi,u_q)) \leq \alpha (V(\xi)) .
\end{equation}
\end{definition}

\begin{remark}\label{rem:fsCLF}
If the conditions in Definition~\ref{def:fsCLF} are satisfied with $M=1$,
we call $V$ a \emph{control Lyapunov function} (\CLF).
We note that this definition of a \CLF differs from the usual definition in the literature by the assumption that the Lyapunov function comes together with an admissible feedback. This is equivalent to the fact that the control value $u$ realizing the decrease of the Lyapunov function satisfies the constraint $f(x,u)\in\X$, because once such a control value exists, the existence of a --- possibly discontinuous ---  admissible feedback is immediate.
In this sense, Definition~\ref{def:fsCLF} extends the definition of a CLF.

In the case $M=1$, the understanding of a \CLF is the following: The existence of a \CLF ensures the existence of an admissible feedback control law for which the resulting \CLF is a Lyapunov function, implying asymptotic $\omega$-stability of system~\eqref{eq:controlled-system}.\\
Definition~\ref{def:fsCLF} now demands that the same is true for $M>1$, a similar reasoning applies: The existence of a \fsCLF ensures
 the existence of an \afcl for which the resulting \fsCLF is a
 \emph{finite-step} Lyapunov function, again implying asymptotic
 $\omega$-stability, see~\cite{Geiselhart.2014c}.
\end{remark}

Similarly as for (1-step) \CLFs also the existence of a \fsCLF yields asymptotic stability as shown next.

\begin{proposition}\label{prop:P}
Let $V:\Rn\to\Rp$ be a fs-CLF for measurement function $\omega:\Rn\to\Rp$.
Let $q$ be the \afcl associated to~$V$.
Then $q$ asymptotically $\omega$-stabilizes the level set ${\cal A}$ for system~\eqref{eq:uncontrolled-system}.
\end{proposition}

\begin{proof}
The invariance of the set $\X$, i.e. $g^i \big(x,q(x)\big)\in \X$ for all  $x \in \X$ and all $i=1,\ldots,M$, is ensured by the \afcl $q$ by  definition.
With this observation, the asymptotic $\omega$-stability of system~\eqref{eq:controlled-system} follows directly from~\cite[Theorem~7]{Noroozi.2018.SG}.
\end{proof}

\subsection{Measurement functions as \fslfcs} \label{sec:measurement-functions-as-fsfcs}

As stated in Proposition~\ref{prop:P},
system~\eqref{eq:uncontrolled-system} is asymptotically
$\omega$-stabilized 
in ${\cal A}$ if a \fsCLF and its associated \afcl $q$ are given.
Generally speaking, it is an open problem to find a (finite-step) (control) Lyapunov function candidate $V$.
Most existing converse Lyapunov theorems for nonlinear systems are not \emph{constructive} in the sense that the results are not usually useful for control purposes.
Recently, constructive converse Lyapunov theorems have been introduced in the case of asymptotic stability with respect to the origin in~\cite[Theorem 13]{Geiselhart.2014c}.
Here we extend Theorem~13 in~\cite{Geiselhart.2014c} to the case of asymptotic stability with respect to closed sets.
Our results show that, under a certain condition, 
the measurement function itself is a finite-step control Lyapunov function for the system.

\begin{theorem}\label{thm:converse-Lyap}
Consider system~\eqref{eq:uncontrolled-system}
with measurement function $\omega:\Rn \to \Rp$, $M\in\N^*$ and an admissible finite-step feedback $q:\X\to\U^M$.
Assume that the resulting closed-loop system~\eqref{eq:controlled-system} satisfies~\eqref{eq:asymp-estimate} with 
\begin{align} \label{eq:beta-condition}
\beta (r,M) < r
\end{align}
for all $ r>0$.
Then the function $V : \Rn \to \Rp$ defined by
\begin{align} \label{eq:converse-Lyap}
V(\xi) := \omega (\xi) \qquad \forall \xi \in \Rn ,
\end{align}
is a finite-step control Lyapunov function for the time step $M$ for~\eqref{eq:uncontrolled-system} with $\ul\alpha=\ol\alpha=\id$ and $\alpha(r)=\beta (r,M)$.
\end{theorem}
\begin{proof}
This is proved using the same arguments as those in the proof
of~\cite[Theorem 13]{Geiselhart.2014c}.
\end{proof}

Theorem~\ref{thm:converse-Lyap} states that, under condition~\eqref{eq:beta-condition}, a measurement function is a \fsCLF.
It is not hard to see that condition~\eqref{eq:beta-condition} \emph{always} holds for exponentially stable systems.
Moreover, there exist systems which are not exponentially stable, but only asymptotically stable and satisfy condition~\eqref{eq:beta-condition} (cf.~\cite[Example 16]{Geiselhart.2014c} for more details). 
Theorem~\ref{thm:converse-Lyap} can, therefore, be used for controller
design: Assume that system~\eqref{eq:uncontrolled-system} is
asymptotically $\omega$-stabilized by a feedback $q$ in ${\cal A}$.
Motivated by Theorem~\ref{thm:converse-Lyap}, one can take $\omega$ as the \fsCLF.
In particular, if system~\eqref{eq:uncontrolled-system} is exponentially
$\omega$-stabilizable in ${\cal A}$, then $\omega$ is always a \fsCLF for
the system and only $M$ needs to be determined.

\color{black}


\section{fs-CLF-Based MPC Approaches}\label{sec:fsclf-based-mpc-approaches}

This section elaborates how to construct stabilizing feedback laws via fs-CLFs.
In particular, we reformulate the control problem into an optimization problem which can be solved efficiently.

\subsection{\fsCLF-based contractive multi-step MPC}

To derive an optimization-based controller design, we impose the following problem.

\begin{problem}\label{prob:OCP-1}
Consider system~\eqref{eq:uncontrolled-system}.
Let $\omega$ be a measurement function.
Let $M\in \N^*$ and a \fsCLF $V$ for the time step $M$ with the associated decay function $\alpha\in\Kinf$, $\alpha<\id$ be given.
Also, let $x(0) =: \xi \in \X$ be given.
Compute 
${\bf u}_M^*=\big( u_0^*, \ldots, u_{M-1}^* \big) \in \U^M$
as an optimal solution of the following \emph{optimal control problem} 
\begin{equation} \label{eq:OCP1} \tag{OCP-1}
\begin{split}
\min_{\textbf{u}_M=(u_0, \ldots,u_{M-1})} \, & \, \sum_{i=0}^{M-1} V\big(x(i,\xi,{\bf u}_M)\big)\\
\mathrm{s.t.} & \mathrm{\, for \, all \, } j \in \{0,...,M-1\}  \\
  & \left\{ \begin{array}{rcll} 
x(j+1)&=&g(x(j),u_{j})  \\
u_j&\in& \U \\
g (x(j),u_{j}) &\in& \X  \\
V(x(M,\xi,{\bf u}_M)) &\leq& \alpha (V(\xi)) .
\end{array} \right. 
\end{split}
\end{equation}
\end{problem}

We note that under our general assumptions an optimal input ${\bf u}_M^*$ need not exist for OCP-1.
A minimal requirement is controlled invariance of ${\cal X}$, which we tacitly assume from now on.
Even then the existence of ${\bf u}_M^*$ is not guaranteed.
In the sequel, we will assume this existence for the sake of simplicity.
Otherwise similar arguments can be applied using approximately optimal inputs.
A similar comment holds for the optimal control problems we formulate below.

Note that $M$ in OCP-1 also determines the optimization horizon of the problem.
Here we make use of the optimal control sequence obtained from OCP-1 as an \afcl.
This implies that the controller communicates with the sensor every $M$ time steps and generates an optimal control sequence of length $M$ by solving OCP-1.
Then the \emph{whole} optimal control sequence is applied to the system and the procedure is repeated.
This procedure is summarized by the following algorithm.

\begin{algorithm}\label{alg:multi-step-mpc}
At each time step $t = k M$, $k \in\N$:
\begin{enumerate}[1)]
\item Measure the state $x(t) \in \X$ of system~\eqref{eq:uncontrolled-system}.
\item Set $\xi := x(t)$, solve Problem~\ref{prob:OCP-1} and denote the optimal control sequence satisfying~\eqref{eq:OCP1} by~${\bf u}_M^*$.
\item Define the finite-step feedback control value $\hat q (\xi)$ by
\begin{align} \label{eq:multi-step-mpc-control}
\hat q (\xi) := {\bf u}_M^* (\xi)
\end{align}
and apply it to system~\eqref{eq:uncontrolled-system} on the time interval
$k M, \ldots, (k+1)M-1$,
\item Go to Step 1.
\end{enumerate}
\end{algorithm}

We note that that the map $\xi \mapsto \hat {q}(\xi)$ implicitly defined in~\eqref{eq:multi-step-mpc-control} is an admissible feedback.
The following lemma shows that even small perturbations of such a feedback are admissible, which accommodates computational errors that are to be expected in applications.
  \begin{lemma}
      \label{lem:feedback-admissible}
Let $\omega:\Rn\to\Rp$ be a measurement function.
Assume $V:\Rn\to\Rp$ is a \fsCLF with associated \afcl $q$.
Then a feedback $h:\X \to \U^M$ is admissible, if it satisfies the constraints of OCP-1 and if in addition
\begin{equation}
    \label{eq:1}
    \sum_{i=0}^{M-1} V\big(x(i,\xi,h(\xi))\big) \leq \sum_{i=0}^{M-1}
    V\big(x(i,\xi,q(\xi))\big) \quad \forall \xi \in \X.
\end{equation}
  \end{lemma}

  \begin{proof}
      The requirement of invariance of $\X$ is part of the assumption, so
      that we only need to check ${\cal K}$-boundedness of the maps $g^j$
      for $h$. To this end note that for $j=1,\ldots,M-1$ we have
      \begin{multline*}
          \omega(g^j(\xi,h(\xi))) \leq \underline{\alpha}^{-1}\circ
          V(g^j(\xi,h(\xi))) \\ \leq \underline{\alpha}^{-1} \left(\sum_{i=0}^{M-1}  
          V(g^i(\xi,h(\xi))) \right) 
         \leq \underline{\alpha}^{-1} \left(\sum_{i=0}^{M-1}  
          V(g^i(\xi,\hat q(\xi))) \right) \\ \leq \underline{\alpha}^{-1}
        \left( \sum_{i=0}^{M-1}
          \overline{\alpha} \circ \omega
          (g^i(\xi,\hat q(\xi))) \right)  \leq \underline{\alpha}^{-1}
        \left(\sum_{i=0}^{M-1}
           \overline{\alpha}
          \circ\kappa_i(\omega(\xi)) \right),
      \end{multline*}
where the $\kappa_i$ are the functions guaranteed by~\eqref{eq:K-bounded-closed-loop} for the admissible feedback $\hat q$.
Finally, for $j=M$ and all $\xi \in \X$ we have using the constraints of OCP-1 that
\begin{equation*}
\omega(g^M(\xi,h(\xi))) \leq \underline{\alpha}^{-1}\circ V(g^M(\xi,h(\xi))) \leq \underline{\alpha}^{-1}\circ \alpha
          \circ \overline{\alpha}(\omega(\xi)).
\end{equation*}
This shows the assertion.
\end{proof}

Now we show that solving Problem~\ref{prob:OCP-1} provides  an \afcl which renders system~\eqref{eq:uncontrolled-system} asymptotically $\omega$-stable.

\begin{proposition}\label{prop:OCP}
    Consider system~\eqref{eq:uncontrolled-system}  and let a measurement
      function $\omega:\Rn\to\Rp$ as well as a \fsCLF $V:\Rn\to\Rp$ be
      given.  Let $q$ be the \afcl associated with the \fsCLF $V$.  Then
      the \afcl~\eqref{eq:multi-step-mpc-control} obtained from
      Algorithm~\ref{alg:multi-step-mpc} yields an admissible feedback
      which asymptotically $\omega$-stabilizes the set ${\cal A} :=
      \omega^{-1}(0)$.
\end{proposition}

\begin{proof}
The feasibility of the Problem~\ref{prob:OCP-1} is guaranteed by the existence of the \afcl $q$ generated by the \fsCLF $V$ and our standing assumption that maximizing arguments exist.
It follows from~\eqref{eq:OCP1} that for all $\xi \in \X$
\begin{equation} 
\label{eq:suffOCP}
\sum_{i=0}^{M-1} V\big(x(i,\xi,{\bf u}_{\hat q})\big) \leq \sum_{i=0}^{M-1} V\big(x(i,\xi,{\bf u}_q)\big) 
\end{equation}
and Lemma~\ref{lem:feedback-admissible} the feedback defined by
\eqref{eq:multi-step-mpc-control} is admissible.
Take any $\xi \in \X$.
For any $t = k M + j$, $k \in \N$, $j \in \{ 0 ,\dots, M - 1 \}$ we have
\begin{align}
 x(t,\xi,{\bf u}_{\hat q}) & =  x(j,x(k M,\xi, {\bf u}_{\hat q}),{\bf u}_{\hat q}(\cdot+kM)) . \label{eq:OCP-proof-1}
\end{align}
With \eqref{eq:suffOCP} we obtain
\begin{align} \label{eq:OCP-proof-2}
\begin{split}
&\sum_{i=0}^{M-1} V\big(x(i,x(k M,\xi,{\bf u}_{\hat q}),{\bf u}_{\hat q})\big) \\
&\hspace{2cm}\leq \sum_{i=0}^{M-1} V\big(x(i,x(k M,\xi,{\bf u}_{\hat q}),{\bf u}_q)\big) .
\end{split}
\end{align}
Moreover, it follows from~\eqref{eq:K-bounded-closed-loop} and~\eqref{eq:3} that
\begin{align}
&\sum_{i=0}^{M-1} V\big(x(i,x(k M,\xi,{\bf u}_{\hat q}),{\bf u}_q)\big) \nonumber\\
& \hspace{1cm}\leq M \max\Big\{ \ol\alpha\big(\omega(x(k M,\xi,{\bf u}_{\hat q}))\big) , \nonumber\\
& \hspace{2cm}\max_{i\in\{1,\dots,M-1\}} \ol\alpha\circ\kappa_i \big( \omega\big(x(k M,\xi,{\bf u}_{\hat q})\big) \big) \Big\} \nonumber\\
& \hspace{1cm}=: \kappa \big( \omega \big(x(k M,\xi,{\bf u}_{\hat q})\big) \big) . \label{eq:OCP-proof-3}
\end{align}
It follows from~\eqref{eq:OCP-proof-2} and~\eqref{eq:OCP-proof-3} that for all $ i \in \{0,\dots,M-1\}$
\begin{align}\label{eq:upper-bound-sum}
\sum_{i=0}^{M-1} V\big(x(i,x(k M,\xi,{\bf u}_{\hat q}),u_{\hat q})\big) \leq  \kappa \big( \omega \big(x(k M,\xi,{\bf u}_{\hat q})\big) \big) .
\end{align}
It follows from the first inequality of~\eqref{eq:3} that for all $ i \in
\{0,\dots,M-1\}$
\begin{align*}
\omega\big(x(i,x(k M,\xi,{\bf u}_{\hat q}),{\bf u}_q)\big) & \leq \ul\alpha^{-1}\circ \kappa \big( \omega \big(x(k M,\xi,{\bf u}_{\hat q})\big) \big) \\
& =: \gamma \big( \omega \big(x(k M,\xi,{\bf u}_{\hat q})\big) \big).
\end{align*}
For $M>0$ we now denote by $\alpha^{1/M}:=\chi\in\Kinf$ a fixed solution of the equation $\chi^M=\alpha$, which exists by~\cite[Proposition~3.1]{Geiselhart.2014b}, though it may not be unique.
Then for $t\geq 0$, the function $\alpha^{t/M} \in \Kinf$ is the $t$-fold composition of $\chi$. As $\alpha < \id$, it follows that $\alpha^{1/M} < \id$, because the condition $\alpha^{1/M}(r)\geq r$ leads by induction to $\alpha^{(t+1)/M}(r)\geq \alpha^{t/M}(r) \geq r$, $t\in \N$. But the latter condition for $t=M$ implies that $\alpha(r) \geq r$, whence $r=0$.

Now as $\alpha^{1/M} < \id$, it follows for all $r>0$ that the map $t\mapsto \alpha^{t/M}(r), t \in \N$ is strictly decreasing to $0$ as $t\to \infty$.
As the map is strictly decreasing we may interpolate linearly in each interval $[t,t+1], t\in \N$, to obtain a strictly decreasing map defined on all of $[0,\infty)$.
With slight abuse of notation we continue to call this map $\alpha^{\cdot/M}(r)$. With this convention, the function $(r,t) \mapsto \alpha^{t/M}(r)$ is in $\mathcal{KL}$.
Also with the decomposition $t=k M+j$, $j\in \{ 0, \ldots, M-1 \}$, we obtain that
\[ \alpha^{t/M}\circ\alpha^{-1} = \alpha^k\circ\alpha^{-(M-j)/M} \ge\alpha^k.\]
From the last two inequalities we can conclude 
\begin{align*}
\omega \big(x(t,\xi,{\bf u}_{\hat q})\big) & \leq \gamma \big(\omega(x(k M,\xi,{\bf u}_{\hat q}))\big) \\
& \leq \gamma \circ \ul\alpha^{-1} \big( V( x(k M,\xi,{\bf u}_{\hat q}) ) \big) \nonumber\\
& \leq \gamma \circ \ul\alpha^{-1} \circ \alpha^k \big( V(\xi) \big) \nonumber\\
& \leq \gamma \circ \ul\alpha^{-1} \circ \alpha^k \circ \ol\alpha \big( \omega (\xi) \big) \\
& \leq \gamma \circ \ul\alpha^{-1} \circ \alpha^{\frac{t}{M}} \circ \alpha^{-1} \circ \ol\alpha \big( \omega ( \xi ) \big) \\
& =: \beta \big(\omega(\xi) , t\big) .
\end{align*}
It is easy to see that $\beta \in \KL$, as $\alpha^{\cdot/M}(\cdot)$ is.
See also~\cite[Lemma 4.2]{Khalil.2002} for a discussion of the necessary details.
\end{proof}

We note that one has to make some standard convexity assumption on the dynamics $g$ to guarantee OCP-1 is numerically solvable via existing algorithms.
We emphasize that OCP-1 needs \emph{no} knowledge of an admissible control $q$.
The difficultly in the computation of ${\bf u}_{\hat q}$ via OCP-1 is,
however, the need for the knowledge of a \fsCLF beforehand and the choice
of a suitable time-step.
As discussed in Section~\ref{sec:measurement-functions-as-fsfcs}, a \fsCLF
candidate can be chosen as the corresponding measurement function for
which only the time-step $M$ remains to be determined.

\subsection{\fsCLF-based contractive updated multi-step MPC}\label{sec:re-optimization}

An obvious drawback of the control scheme proposed by Proposition~\ref{prop:OCP} is that it only communicates with the sensor every $M$ time-steps.
Hence, the control loop is closed less often than that for a classic closed-loop control, which may make the system less robust with respect to perturbations.
As shown in~\cite{Grune.2015,Grune.2014b}, a remedy to this problem is to re-compute the remaining part of the optimal control sequence at each time instant.
This amounts to solving an optimal control problem with shortened horizon.

\begin{problem}\label{prob:OCP-2}
Consider system~\eqref{eq:uncontrolled-system}.
Let a measurement function $\omega$,  $M\in \N^*$ and a \fsCLF $V:\Rn\to\Rp$ with associated decay function $\alpha\in\Kinf$, $\alpha<\id$ be given.
Furthermore, let $j \in \{1,\ldots,M\}$.
For a given initial value $\tilde \xi \in \X$ consider a control sequence
$ \mathbf{\tilde u}=(\tilde u_0,\ldots, \tilde u_{M-j-1})$  satisfying
$x(M-j,\tilde \xi, \mathbf{\tilde u}) \in \X$. Define $x(0) = \xi:= x(M-j,\tilde \xi, \mathbf{\tilde u})$.
Compute ${\bf u}_j=\big( u_1 (\xi), \ldots, u_j (\xi) \big)$
as the optimal solution the following \emph{optimal control problem} 
\begin{equation} \label{eq:OCP2} \tag{OCP-$2_j$}
\begin{split}
& \min_{\textbf{u}_j=(u_j(0), \ldots, u_j(j-1))} \, \, \sum_{i=0}^{j-1} V\big(x(i,\xi,{\bf u}_j)\big)\\
\mathrm{s.t.} & \mathrm{\, for \, all \, }  \ell \in \{0,...,j-1\} \\
&  \left\{ \begin{array}{rcll} 
x(\ell+1)&=&g(x(\ell),u_{j}(\ell))  \\
u_j(\ell)&\in& \U \\
g (x(\ell),u_j(\ell)) &\in& \X  \\
V(x(j,\xi,{\bf u}_j)) &\leq& \alpha (V(\tilde \xi)) . 
\end{array} \right. 
\end{split}
\end{equation}
\end{problem}

Note that feasibility of Problem~\ref{prob:OCP-2} depends, among others, on the initial control sequence $\mathbf{\tilde u}$. However, it is not hard to see that if we consider a control sequence $\mathbf{\hat u}=(\hat u_1, \ldots, \hat u_M)$ solving Problem~\ref{prob:OCP-1}, then for any $j=1,\ldots,M$ and initial control sequence $\mathbf{\tilde u}:= (\hat u_1, \ldots, \hat u_{M-j})$ a solution of Problem~\ref{prob:OCP-2} is given by $\mathbf{u_j}=(\hat u_{M_j+1}, \ldots, \hat u_{M})$.
The idea is to iteratively solve Problem~\ref{prob:OCP-2} and only to apply the first control value to shrink the horizon by one.
The algorithm for such a control strategy is formalized as follows.

\begin{algorithm} \label{alg:re-opt-step-mpc}
At each time step $t = k M+j$, $k\in \N$, $j=0,\ldots,M-1$:
\begin{enumerate}[1)]
\item Measure the state $x(t) \in \X$ of system~\eqref{eq:uncontrolled-system}.
\item Set $\xi := x(t)$, $j=j(t)$ and solve Problem~\ref{prob:OCP-2} and denote the optimal control sequence satisfying~\eqref{eq:OCP2} by ${\bf u}_j^*$.
\item Set the control value to
\begin{align} \label{eq:re-opt-mpc-control}
u(k)=\mathbf{u}_j^*(0)
\end{align}
and apply it to system~\eqref{eq:uncontrolled-system} at time $t=k M+j$.
\item Go to Step 1.
\end{enumerate}
\end{algorithm}

\begin{remark}\label{rem:OP}
We note that by the optimality principle~\cite[Corollary 3.16]{Grune.2017a} the solutions to Algorithm~\ref{alg:multi-step-mpc} and Algorithm~\ref{alg:re-opt-step-mpc} \emph{coincide} in the absence of perturbations.
\end{remark}

To illustrate the two proposed algorithms, we give an example.

\begin{example}  \label{ex:mstep_mpc}

Here we consider an illustrative numerical example for which we compare the two different MPC approaches. Since both algorithms produce identical results in the case without perturbations, we compare them for the situation in which the controller is derived by optimizing over the nominal, i.e., unperturbed system but then applied to a perturbed system. We consider the nominal system described by
\begin{align}\label{eq:nominal}
\begin{array}{rcl}
x_1^+ &=&x_1+x_2 \\
x_2^+ &=&x_2+x_3 \\
x_3^+ &=& \frac{3}{2} x_3+u 
\end{array}
\end{align}
and the corresponding perturbed system
\begin{align}\label{eq:perturbed}
\begin{array}{rcl}
x_1^+ &=&x_1+x_2 +0.1\sin(k/4) \\
x_2^+ &=&x_2+x_3 \\
x_3^+ &=& \frac{3}{2} x_3+u .
\end{array}
\end{align}
Note that the nominal system~\eqref{eq:nominal} is open-loop unstable. 

Motivated by the converse Lyapunov function result in Theorem~\ref{thm:converse-Lyap} we start by considering the candidate \fsCLF $V(x)=x^\top P x$ with
$$
P = \left(\begin{array}{ccc}
1 &  0 & 0.25 \\
0  & 1 & 0.25 \\
0.25 & 0.25 & 1
\end{array}\right) ,
$$
which is obviously of the form~\eqref{eq:converse-Lyap}. 
This choice of the matrix $P$ contains cross terms between the states. 
It is easy to check that the function $V$ thus defined is an $M$-step Lyapunov function for $M=3$.
However, in order to obtain more pronounced differences between Algorithms~\ref{alg:multi-step-mpc} and \ref{alg:re-opt-step-mpc}, we used $M=6$ in the simulations.
Moreover, we used $\alpha(r)=0.9r$ in both Problem \ref{prob:OCP-1} and \ref{prob:OCP-2} and all simulations were performed with the initial condition $\xi=(-1, 1, 1)^T$.

Figure~\ref{fig:nominal} illustrates the state trajectories corresponding to the nominal case for Algorithm~\ref{alg:multi-step-mpc}.

\begin{figure}[!ht]
\centering
\includegraphics[width=0.49\textwidth]{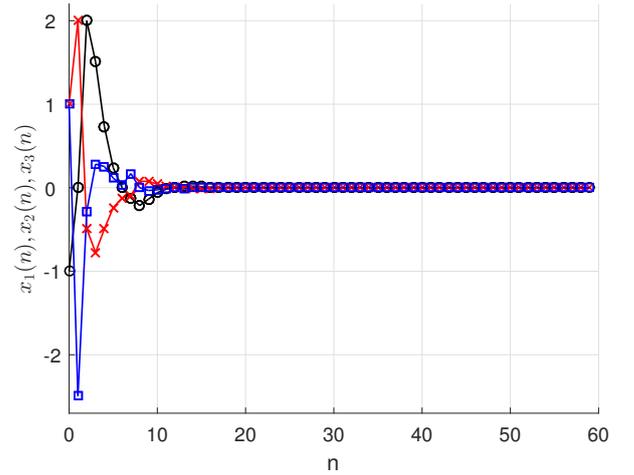}
\vspace{-0.6cm}
\caption{State trajectories of the nominal system~\eqref{eq:nominal}, $x_1$ (black~$\pmb\circ$), $x_2$ (red ~$\color{red}\pmb\times$), $x_3$ (blue $\color{blue}\pmb\square$) with control input computed via Algorithm~\ref{alg:multi-step-mpc}. }
\label{fig:nominal}
\end{figure}

The case in which the control sequence computed by
Algorithm~\ref{alg:multi-step-mpc} is applied to the perturbed
system~\eqref{eq:perturbed} is depicted by Figure~\ref{fig:perturbed}. One
clearly sees that the $x_1$-component, in which the perturbation enters in
\eqref{eq:perturbed}, is more strongly affected by the perturbation than the other components of the solution.
\begin{figure}[!ht]
\centering
\includegraphics[width=0.49\textwidth]{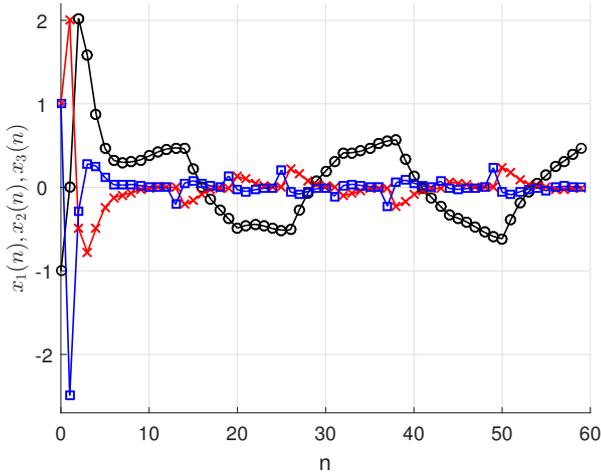}
\vspace{-0.6cm}
\caption{State trajectories of the perturbed system~\eqref{eq:perturbed}, $x_1$ (black~$\pmb\circ$), $x_2$ (red $\color{red}~\pmb\times$), $x_3$ (blue~$\color{blue} \pmb\square$) with control input computed via Algorithm~\ref{alg:multi-step-mpc}. }
\label{fig:perturbed}
\end{figure}

Finally, Figure~\ref{fig:updated} illustrates the state trajectories
associated with the shrinking horizon strategy with re-optimization, i.e., Algorithm~\ref{alg:re-opt-step-mpc} applied to the perturbed system~\eqref{eq:perturbed}.
It may be observed that the re-optimization on shrinking horizons is able to mitigate the effect of the perturbation, as the maximal deviation of the $x_1$-component from the desired equilibrium $x_1=0$ after the transient phase is reduced by about 37\%, from 0.615 to 0.387.
\begin{figure}[!ht]
\centering
\includegraphics[width=0.49\textwidth]{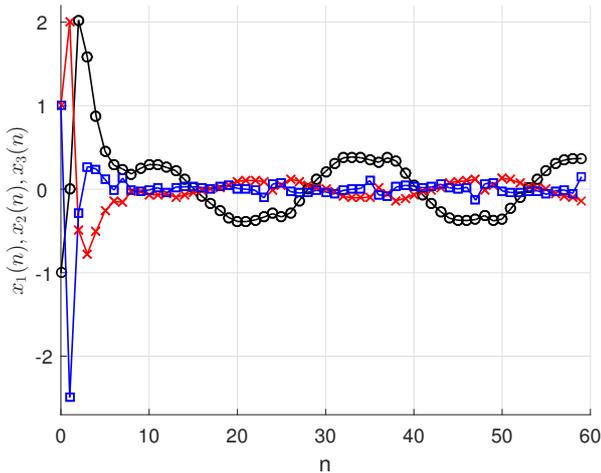}
\vspace{-0.6cm}\centering
\caption{State trajectories of the perturbed system~\eqref{eq:perturbed}, $x_1$ (black~$\pmb\circ$), $x_2$ (red~$\color{red}\pmb\times$), $x_3$ (blue~$\color{blue}\pmb\square$) with control input computed via Algorithm~\ref{alg:re-opt-step-mpc}. }
\label{fig:updated}
\end{figure}
\end{example}

\subsection{\fsCLF-based classic MPC} \label{sec:classic-MPC-without-terminal-constraints}

The shrinking horizon method is a rather unusual way of obtaining a feedback law via optimization based techniques.
More commonly, one would use a classic MPC approach, in which the optimization is performed at every time step over a \emph{fixed} horizon length $N$ and always the \emph{first} element of the resulting control sequence is implemented.
In this section we show that this approach can also be applied using fs-CLFs.
To this end, we consider the following optimal control problem.

\begin{problem}\label{prob:OCP-3}
Consider system~\eqref{eq:uncontrolled-system} and let $M,N\in \N^*$.
Let $\omega$ be a measurement function and $V$ be a \fsCLF with the associated decay function $\alpha\in\Kinf$, $\alpha<\id$ be given.
Also, let $x(0) =: \xi \in \X$.
Compute ${\bf u}_N^*=\big( u_0^* (\xi), \ldots, u_{N-1}^* (\xi) \big)$
as the optimal solution of  the following \emph{optimal control problem} \emph{(OCP-3)}
\begin{equation} \label{eq:OCP3} \tag{OCP-3}
\begin{split}
\min_{\textbf{u}_N} \, & \, \sum_{i=0}^{N-1} V\big(x(i,\xi,{\bf u}_N)\big) \\
\mathrm{s.t.}& \, \mathrm{\, for \, all \, } j \in \{0,\dots,N-1\} \\
 & \left\{ \begin{array}{rcll} 
x(j+1)&=&g(x(j),u_{j})  \\
u_k&\in& \U \\
g (x(j),u_j) &\in& \X  .
\end{array} \right. 
\end{split}
\end{equation}
\end{problem}

Here we make use of the feedback signal at \emph{every} time step.
To do this, one can solve Problem~\ref{prob:OCP-3} every single time-step and apply the \emph{first} element of the corresponding optimal control sequence ${\bf u}_N$ to the system and then the \eqref{eq:OCP3} is solved again.
This procedure is summarized by the following algorithm.
\begin{algorithm} \label{alg:classic-mpc}
At each time step $t \in \N$:
\begin{enumerate}[1)]
\item Measure the state $x(t) \in \X$ of system~\eqref{eq:uncontrolled-system}.
\item Set $\xi := x(t)$, solve Problem~\ref{prob:OCP-3} and denote the optimal control sequence satisfying~\eqref{eq:OCP3} by ${\bf u}_N^*$.
\item Define the MPC-feedback value $\hat q_{\mathrm{MPC}}$ by
\begin{align} \label{eq:mpc-control}
\hat q_{\mathrm{MPC}} (\xi) := u_0^* (\xi)
\end{align}
and apply it to system~\eqref{eq:uncontrolled-system}.
\item Go to Step 1.
\end{enumerate}
\end{algorithm}

The solution to the resulting MPC closed-loop system starting from some
initial value $\xi$ and with optimization horizon $N$ is denoted by $x_{\mathrm{MPC}(N)}(\cdot,\xi)$.
We denote the \emph{optimal value function} related to Problem~\ref{prob:OCP-3} by
\begin{align}\label{eq:OCP-value-function}
 & V_N (\xi)  := \sum_{i=0}^{N-1} V\big(x(i,\xi,{\bf u}_N^*)\big) . 
\end{align}
In order to analyze the resulting MPC closed-loop system, we make use of the following result.

\begin{definition}
We say that the MPC scheme described in Algorithm~\ref{alg:classic-mpc} is semiglobally practically
asymptotically $\omega$-stabilizing with respect to the optimization horizon
$N$ in ${\cal A} := \omega^{-1}(0)$, if  there exists $\beta \in \KL$ such that the following property holds: for each $\delta> 0$ and $\Delta>\delta$ there exists $N_{\delta,\Delta} \in\N^*$ such that for all
optimization horizons $N\geq N_{\delta,\Delta}$ and all $\xi\in\R^n$ with $\omega(\xi)<\Delta$ the closed-loop solutions $x_{\mathrm{MPC}(N)}(\cdot,\xi)$ satisfy
\begin{equation*}
 \omega\big(x_{\mathrm{MPC}(N)}(t,\xi)\big) \leq \max\left\{\beta\big(\omega(\xi), t \big),\,\delta\right\} \qquad \forall t \in \N .
\end{equation*}
\end{definition}

\begin{proposition}\label{prop:MPCstab}
Let $\omega:\Rn\to\Rp$ be a measurement function and $V$ be a \fsCLF. Assume that there is a $\K_\infty$-function $\sigma$ such that the optimal value function $V_N$ from \eqref{eq:OCP-value-function} satisfies
\[ V_N(\xi) \le \delta(V(\xi)),\quad \forall x\in \X, N\in\N^*. \]
Then the MPC scheme obtained from Algorithm~\ref{alg:classic-mpc} is semiglobally practically asymptotically $\omega$-stabilizing  in ${\cal A} := \omega^{-1}(0)$ for system~\eqref{eq:uncontrolled-system} with respect to the optimization horizon $N$.
If, moreover, $\sigma$ is a linear function, i.e., $\sigma(r)=\gamma r$ for some $\gamma\in\R$, then the resulting MPC closed-loop is asymptotically $\omega$-stable in ${\cal A}:= \omega^{-1}(0)$ for all $N > 2+\frac{\ln(\gamma -1)}{\ln \gamma - \ln (\gamma-1)}$. 
\end{proposition}
\begin{proof}
The first statement is proved by following similar arguments as those in~\cite[Theorem 6.37]{Grune.2017a}.
For the second statement, see~\cite[Corollary 6.21 and Remark 6.22]{Grune.2017a}.
We note that Theorem 6.37 and Corollary 6.21 in~\cite{Grune.2017a} consider stabilization at an equilibrium point.
However, it is not hard to see that with the obvious modification of the arguments in these references we obtain that
\[ V_N\big(x_{\mathrm{MPC}(N)}(t,\xi))\big) \le \max\{\tilde\beta(V_N(\xi),i),\tilde\delta\} \]
for all $\xi$ with $V_N(\xi)\le \widetilde\Delta$, $\widetilde\Delta>0$, $\tilde\beta\in\KL$, and $\tilde\delta\ge 0$, where $\widetilde\Delta\to\infty$ and $\tilde\delta\to 0$ as $N\to\infty$. Moreover, the inequality holds for arbitrarily large $\widetilde\Delta>0$ and $\tilde\delta=0$ if $\sigma$ is linear.
Now, the inequalities 
\[ V(\xi) \le V_N(\xi) \le \delta(V(\xi)),\]
which follow by definition of $V_N$ and by the assumption of the proposition, imply that $V_N$ is a Lyapunov function for the closed loop, which proves (practical) asymptotic stability.
\end{proof}

In order to check whether the optimal value function \eqref{eq:OCP-value-function} satisfies the conditions in Proposition~\ref{prop:MPCstab}, we make the following observation: From \eqref{eq:3} and \eqref{eq:K-bounded-closed-loop} it follows that there exists a $\K_\infty$-function $\hat\kappa$ such that for each admissible finite-step feedback control law $q:\X\to\U^M$ the inequality
\begin{equation} V\Big(g^i\big(\xi,q_1(\xi),\ldots,q_i(\xi)\big)\Big) \le \hat\kappa(V(\xi)) \label{eq:Vtransient}\end{equation}
holds for all $i=1,\ldots,M-1$. This fact is easily verified for $\hat\kappa(r) = \max_{i=1,\ldots,M-1} \overline\alpha\circ\kappa_i\circ\underline\alpha^{-1}(r)$.
Now we give the main result of this section.

\begin{theorem}\label{thm:OCP-3a}
Consider system~\eqref{eq:uncontrolled-system} and let $M,N\in \N^*$.
Let $\omega:\Rn\to\Rp$ be a measurement function and $V$ be a \fsCLF for the step size $M$.
Then, the following statements hold.
\begin{enumerate}
\item[(i)] If in \eqref{eq:decayV} $\alpha(s)=cs$ with $c\in[0,1)$ and in
  \eqref{eq:Vtransient} $\hat\kappa(r)=dr$ with $d>0$, then the MPC closed-loop is asymptotically $\omega$-stable with respect to the optimization horizon $N$ in ${\cal A} := \omega^{-1}(0)$ for all $N > 2+\frac{\ln(\gamma -1)}{\ln \gamma - \ln (\gamma-1)}$ with $\gamma=Md/(1-c)$.
\item[(ii)] If  in \eqref{eq:decayV} $\alpha(s)=cs$ with $c\in[0,1)$ and
  $\hat\kappa$ in \eqref{eq:Vtransient} satisfies $\hat\kappa(r)\le q\max
  \{r^a,r^b\}$ for constants $a,b,q>0$, then the MPC scheme is semiglobally practically asymptotically $\omega$-stabilizing with respect to the optimization horizon $N$ in ${\cal A}$.
\item[(iii)] There exists $\rho\in\K_\infty$ such that if we replace $V$
  by $\widetilde V = \rho(V)$ in Problem~\ref{prob:OCP-3}, then the MPC scheme is semiglobally practically asymptotically $\omega$-stabilizing with respect to the optimization horizon $N$ in ${\cal A}$.
\end{enumerate}
\end{theorem}
\begin{proof}
(i) Iterating \eqref{eq:decayV} yields the existence of a control function ${\bf u}$ satisfying $V(x(kM,\xi,{\bf u})) \le c^k V(\xi) $  and $V(x(kM+j,\xi,{\bf u})) \le d V(x(kM,\xi,{\bf u}))$ for all $k\in\N$ and $j=0,\ldots,M-1$.
Together this yields 
\[ V(x(kM+j,\xi,{\bf u})) \le c^kd V(\xi), \]
which implies for $K\in\N$ such that $KM\ge N$
\begin{eqnarray*} 
V_N(\xi) & \le & \sum_{i=0}^{N-1} V\big(x(i,\xi,{\bf u})\big)  \\
& \leq&  \sum_{k=0}^{K-1} \sum_{j=0}^{M-1} V\big(x(kM+j,\xi,{\bf u})\big)\\
& \le & \sum_{k=0}^{K-1} M c^kd V(\xi) \; \le \; \frac{Md}{1-c} V(\xi).
\end{eqnarray*}
Now the second part of Proposition \ref{prop:MPCstab} with $\delta(r) =
\frac{Md}{1-c} r$ yields the claim.

(ii) Iterating \eqref{eq:decayV}, as in (i) we obtain the inequality
\begin{equation*} 
\begin{split}
V(x(kM+j,\xi,{\bf u})) &\le q \max\{ (c^k V(\xi))^a,(c^k V(\xi))^b\} \\
& \le q\max\{c^a,c^b\}^k\max\{V(\xi)^a,V(\xi)^b\}. 
\end{split}
\end{equation*}
Abbreviating $\hat c= \max\{c^a,c^b\}\in[0,1)$ the same computation as in (i) yields
\begin{equation*} 
\begin{split} V_N(\xi) &\le \sum_{k=0}^{K-1} M q \hat c^k \max\{V(\xi)^a,V(\xi)^b\} \\
&\le \frac{Mq}{1-\hat c}\max\{V(\xi)^a,V(\xi)^b\}. 
\end{split}
\end{equation*}
This implies that the assumptions of the first part of Proposition \ref{prop:MPCstab} are satisfied with $\delta(r) = \frac{Mq}{1-\hat c}\max\{r^a,r^b\}$ and the claim follows.

(iii) Recall $\hat{\kappa}$ from \eqref{eq:Vtransient}. From \eqref{eq:decayV} it follows that 
\[ \hat\kappa ( V(x(M,\xi,{\bf u}_q))) \le \hat\kappa \circ \alpha (V(\xi)) = \underbrace{\hat\kappa \circ \alpha \circ \hat\kappa^{-1}}_{=:\mu}(\hat\kappa (V(\xi)). \]
Since $\alpha<\id$ it follows that $\mu<\id$. Hence, applying Proposition 3.2 from~\cite{Grune.2014} to $\widehat V = \hat\kappa(V)$ implies that there exists $\rho\in\K_\infty$ and $\lambda\in(0,1)$ such that $W := \rho(\widehat V)$ satisfies $W(x(M,\xi,{\bf u}_q)) \le \lambda W(\xi)$. 
Iterating this inequality yields the existence of ${\bf u}$ with
\begin{equation*}
\begin{split}
W(x(kM,\xi,{\bf u})) \le \lambda^k W(\xi) \\ 
 V(x(kM+j,\xi,{\bf u})) \le \hat\kappa(V(x(kM,\xi,{\bf u}))).
 \end{split} 
 \end{equation*}
For $\widetilde V = \rho(V)$ this implies
\begin{equation*}
\begin{split} 
\widetilde V(x(kM+j,\xi,{\bf u})) & =  \rho(V(x(kM+j,\xi,{\bf u}))  \\
&\hspace{-1cm} \le  \rho\circ\hat\kappa(V(x(kM,\xi,{\bf u}))\\
&\hspace{-1cm}  =  W(x(kM,\xi,{\bf u})) \\
&\hspace{-1cm} \le  \lambda^k W(\xi)   =  \lambda^k \underbrace{\rho\circ\hat\kappa\circ\rho^{-1}}_{=:\sigma\in\K_\infty}(\widetilde V(\xi)).
\end{split}
\end{equation*}
With an analogous computation as in (i) and (ii) we obtain
\begin{equation*}
\begin{split}
 V_N(\xi) \le \sum_{k=0}^{K-1} \sum_{j=0}^{M-1} \widetilde V\big(x(kM+j,\xi,{\bf u})\big) \\
\hspace{1cm}\le \sum_{k=0}^{K-1} M \lambda^k \sigma(\widetilde V(\xi)) \le \frac M{1-\lambda}\sigma(\widetilde V(\xi)).
  \end{split} 
 \end{equation*}
Hence, the assumptions of the first part of Proposition \ref{prop:MPCstab} are satisfied for $\widetilde V$ in place of $V$ with $\delta(r) = M \sigma(r)/(1-\lambda)$. Thus, Proposition~\ref{prop:MPCstab} yields the assertion.
\end{proof}

\begin{example} 
We illustrate the performance of the classic MPC approach again for the nominal and perturbed systems~\eqref{eq:nominal} and~\eqref{eq:perturbed}.
We use the same initial condition $\xi=(-1, 1, 1)^T$ as in Example~\ref{ex:mstep_mpc}, the fs-CLF $V(x)=x^TPx$ from Example~\ref{ex:mstep_mpc} as stage cost and the optimization horizon $N=6$. Figure \ref{fig:classic} shows the resulting state trajectory for applying the control computed by Algorithm~\ref{alg:classic-mpc} to the perturbed system~\eqref{eq:perturbed}.
The effect of the perturbation is comparable to the updated shrinking horizon MPC Algorithm in Figure~\ref{fig:updated}; after the transient phase the maximal deviation of $x_1$ to the desired equilibrium is 0.363 here compared to 0.387 in the shrinking horizon algorithm.
However, one observes that the trajectories in Figure~\ref{fig:classic} appear smoother than those in Figure~\ref{fig:updated}.
Further numerical tests have revealed that the loss of smoothness in Figure~\ref{fig:updated} is mainly due to the contractive constraints and not due to the shrinking horizon. Hence, this is an advantage for MPC without using contraction constraints.
However, we emphasize that we have not rigorously checked the assumptions of Theorem~\ref{thm:OCP-3a} (which are usually quite conservative, anyway), but rather determined the optimization horizon $N$ by trial and error.
Hence, in contrast to Algorithms~\ref{alg:multi-step-mpc} and~\ref{alg:re-opt-step-mpc}, there is no formal guarantee for asymptotic stability here. 

\begin{figure}[!ht]
\centering
\includegraphics[width=0.49\textwidth]{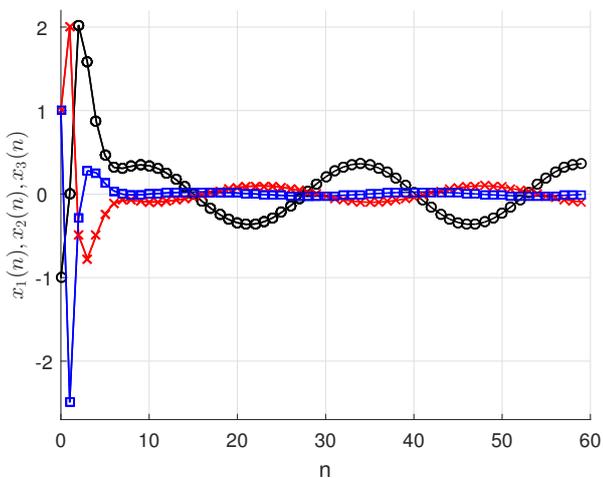}
\vspace{-0.6cm}
\caption{State trajectories of the perturbed system~\eqref{eq:perturbed}, $x_1$ (black~$\pmb\circ$), $x_2$ (red~$\color{red}\pmb\times$), $x_3$ (blue $\color{blue}\pmb\square$) 
with control input computed via Algorithm~\ref{alg:classic-mpc}. }
\label{fig:classic}
\end{figure}
\end{example}

\section{Conclusions and outlook} \label{sec:conclusions}

We have exploited the notion of fs-CLF to develop control design approaches for discrete-time systems. 
To this end, the controller design problem has been reformulated into an optimization problem.
Motivated by state-of-the-art applications, we have provided three different MPC schemes via fs-CLFs: i) contractive multi-step MPC, ii) contractive updated multi-step MPC, iii) classic MPC without stabilizing terminal constraints.
We have illustrated the MPC schemes via an example.

The results of the paper can be extended in several directions: fs-LFs are leveraged to develop \emph{nonconservative} small-gain and dissipativity conditions for stability analysis of large-scale systems~\cite{Noroozi.2014,Geiselhart.2015,Gielen.2015,Noroozi.2018.SG}.
We aim to fuse the results of the current paper with the nonconservative small-gain and dissipativity to develop distributed MPC schemes.
Applications of such results to smart grids, smart city and mobile robots are expected.
The analysis in this work can also be generalized to systems subject to disturbances.

\bibliographystyle{IEEEtran}
\bibliography{RobustConsensusBib}   

\end{document}